\documentclass[12pt]{amsart}

\usepackage{graphicx,color}
\usepackage{url,xspace}
\usepackage[hypertex]{hyperref}
\usepackage{amssymb}
\usepackage{tabularx,ragged2e}
\usepackage{tikz}

\newtheorem{theorem}{Theorem}
\newtheorem{lemma}[theorem]{Lemma}
\newtheorem{corollary}[theorem]{Corollary}

\newtheorem{proposition}[theorem]{Proposition}

\newcommand{\Z}{{\mathbb Z}}
\newcommand{\ZZ}{{\mathbb Z}}

\newcommand{\NN}{{\mathbb N}}
\newcommand{\PP}{{\mathbb P}}
\renewcommand{\P}{{\mathbb P}}

\newcommand{\ind}{\mathbf{1}}
\newcommand{\cI}{\mathcal{I}}
\newcommand{\tT}{{\tilde{T}}}
\newcommand{\tA}{{\tilde{A}}}
\newcommand{\bigmid}{\,\big|\,}
\newcommand{\tor}{\stackrel{r}{\to}}
\newcommand{\Geom}{{\ensuremath{\text{Geom}}}}
\newcommand{\Zodd}{{\mathbb{Z}_{\text{\rm odd}}}}
\newcommand{\Zeven}{{\mathbb{Z}_{\text{\rm even}}}}
\newcommand{\comb}{\mathbb{K}}
\renewcommand{\tilde}{\widetilde}
\newcommand{\df}{\textbf}

\author{Alexander E.\ Holroyd}
\address[Alexander E.\ Holroyd]{\sloppypar Microsoft Research,
1 Microsoft Way, Redmond, WA 98052, USA}
\email{holroyd at microsoft.com}
\urladdr{\url{http://research.microsoft.com/~holroyd}}

\author{James Martin}
\address[James Martin]{\sloppypar University of Oxford,
Department of Statistics, 1 South Parks Road, Oxford OX1 3TG, UK}
\email{martin at stats.ox.ac.uk}
\urladdr{\url{http://www.stats.ox.ac.uk/~martin}}



\title{Stochastic Domination and Comb Percolation}
\date{30 January 2012}

\begin{document}

\begin{abstract}
There exists a Lipschitz embedding of a $d$-dimensional
comb graph (consisting of infinitely many parallel copies
of $\Z^{d-1}$ joined by a perpendicular copy) into the open set
of site percolation on $\Z^d$, whenever the parameter $p$ is
close enough to $1$ or the Lipschitz constant is sufficiently large.
This is proved using several new results
and techniques involving stochastic domination, in contexts
that include a process of independent overlapping intervals on
$\Z$, and first-passage percolation on general graphs.
\end{abstract}

\maketitle

\section{Introduction}

The following natural generalization of percolation theory is
prompt\-ed by the results of \cite{DDGHS,GH-embed}. Let $G$ and
$H$ be graphs. For $p\in[0,1]$, consider the site percolation
model on $H$, in which each vertex is \df{open} with
probability $p$, and otherwise \df{closed}, independently for
different vertices. An \df{embedding} of $G$ in the open set of
$H$ is an injective map from the vertex set of $G$ to the set
of open vertices of $H$, such that neighbours in $G$ map to
neighbours in $H$. Define the critical probability
\begin{multline*}
p_c(G,H):=\\ \inf\Bigl\{p:\P\bigl(\exists \text{ an embedding
of $G$ in the open set of $H$}\bigr)>0\Bigr\}.
\end{multline*}

If $\Z_+$ is a singly-infinite path then $p_c(\Z_+,H)$ is simply the usual
critical probability $p_c(H)$ of site percolation on $H$ (see e.g.\ \cite{g2}
for background).  For the doubly-infinite path $\Z$, it was proved in
\cite[Proof of Theorem 3.9]{lyons} that $p_c(\Z,H)$ also equals $p_c(H)$ for
any infinite connected $H$.  Observe that if $G,H$ are subgraphs of $G',H'$
respectively then $p_c(G,H')\leq p_c(G',H)$.

We focus on the question: for which graphs is it the case that
$p_c(G,H)<1$? Let $\Z^d$ be the usual cubic lattice, with
vertex set also denoted $\Z^d$, and with vertices $x,y$ joined
by an edge whenever $\|x-y\|_1=1$.  Also let $\Z^d_{[M]}$
denote the spread-out lattice, in which vertices $x,y\in\Z^d$
are joined whenever $0<\|x-y\|_\infty\leq M$.  It was proved in
\cite{DDGHS} and \cite{GH-embed} respectively that
$p_c(\Z^{d-1},\Z^d_{[2]})<1$, while on the other hand
$p_c(\Z^d,\Z^d_{[M]})=1$ for all $M$.

An embedding of $G$ into $\Z^d_{[M]}$ may also be regarded as an
$M$-Lipschitz embedding of $G$ into $\Z^d$. In that language, the results
mentioned in the previous paragraph say that $M$-Lipschitz embeddings of
$\Z^{d-1}$ into $\Z^d$ are possible whenever $p$ or $M$ is large enough,
while Lipschitz embeddings of $\Z^d$ into $\Z^d$ are never possible for
$p<1$.

In this article we address a case lying between the last two mentioned above.
Define the $d$-dimensional \df{comb graph} $\comb^d$ to have vertex set
$\Z^d$, and edges $(z, z+e_i)$ for every $z\in\ZZ^d$ and all $i$ with $1\leq
i\leq d-1$, together with $(z, z+e_d)$ for all $z$ such that $z_1=0$ (where
$e_1, \dots, e_d$ are the standard basis vectors). Thus, $\comb^d$ consists
of a stack of parallel copies of $\ZZ^{d-1}$ (perpendicular to coordinate
$d$), connected by a single perpendicular copy of $\ZZ^{d-1}$ (perpendicular
to coordinate $1$). For $d>2$, $\comb^{d}$ is isomorphic to the product of
the $2$-dimensional comb $\comb^2$ with $\Z^{d-2}$. See Figure~\ref{radiator}
for illustrations of $\comb^2$ and $\comb^3$.
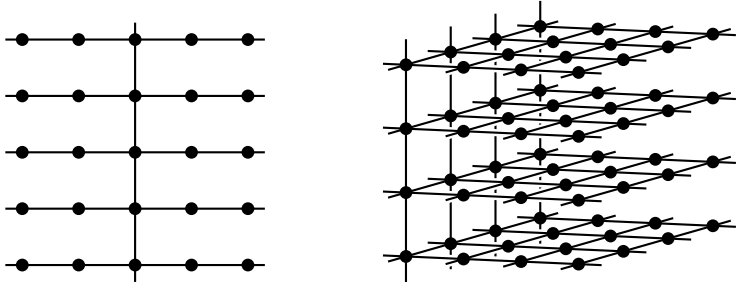
\begin{figure}
\centering
\hfill
\begin{tikzpicture}[thick,scale=0.75]
\tikzstyle{every node}=[shape=circle,fill=black,inner sep=0,minimum size=.17cm]
\foreach \x in {2,...,2} \draw (\x,-.3)--(\x,4.3);
\foreach \y in {0,...,4} \draw (-.3,\y)--(4.3,\y);
\foreach \x in {0,...,4} \foreach \y in {0,...,4} \node at (\x,\y) {};
\end{tikzpicture}
{}\hfill
\begin{tikzpicture}[x={(.9cm,-.04cm)},y={(.7cm,.2cm)},z={(0cm,1cm)},thick,scale=0.85]
\tikzstyle{every node}=[shape=circle,fill=black,inner sep=0,minimum
size=.17cm]
\newcommand{\drawbg}[2]
{\draw[white,line width=2.7pt,opacity=1.0]  (#1) -- (#2); }
\foreach \x in {0,...,0} \foreach \y in {0,...,3} \draw
(\x,\y,-.4)--(\x,\y,3.4);
\foreach \x in {0,...,3} \foreach \z in {0,...,3} \drawbg
{\x,-.4,\z}{\x,3.4,\z};
\foreach \z in {0,...,3} \foreach \y in {0,...,3} \drawbg
{-.4,\y,\z}{3.4,\y,\z};
\foreach \x in {0,...,3} \foreach \z in {0,...,3} \draw
(\x,-.4,\z)--(\x,3.4,\z);
\foreach \z in {0,...,3} \foreach \y in {0,...,3} \draw
(-.4,\y,\z)--(3.4,\y,\z);
\foreach \x in {0,...,3} \foreach \y in {0,...,3} \foreach \z in {0,...,3}
\node at (\x,\y,\z) {};
\end{tikzpicture}
\hfill {}
\caption{Part of the comb graph $\comb^d$ in dimension $d=2$
(left) and $d=3$ (right).} \label{radiator}
\end{figure}

\begin{theorem}[Comb percolation]\label{pc}
We have $p_c(\comb^d,\Z^d_{[2]})<1$ for all $d\geq 2$.
\end{theorem}

\begin{corollary}\label{large-m}
For all $d\geq 2$ we have $p_c(\comb^d,\Z^d_{[M]})\to 0$ as
$M\to\infty$.
\end{corollary}

Our proof gives an explicit upper bound for $p_c(\comb^d,\Z^d_{[2]})$, but we
have not attempted to optimize it.  The spread-out lattice $\Z^d_{[2]}$ in
Theorem~\ref{pc} cannot be replaced with the nearest-neighbour lattice
$\Z^d$.  Indeed, it was proved in \cite{GH-embed} that $p_c(\Z^2,\Z^d)=1$ for
all $d\geq 2$; since $\Z^2$ is a subgraph of $\comb^d$ this implies
$p_c(\comb^d,\Z^d)=1$ for $d\geq 3$.  It is also easy to see that
$p_c(\comb^2,\Z^2)=1$, since the backbone of $\comb^2$ would have to be
embedded as a straight line in $\Z^2$.  On the other hand, our techniques may
be adapted to prove $p_c(\comb^d,H)<1$ for some graphs $H$ with edge sets
intermediate between those of $\Z^d$ and $\Z^d_{[2]}$ -- in particular it
seems plausible that this could be done for the ``star lattice''
$\Z^d_{[1]}$, but we have not pursued this. Such questions reflect details of
the local lattice geometry, whereas the fact that $p_c(\comb^d,\Z^d_{[M]})<1$
for large enough $M$ (as implied by Theorem \ref{pc}) is more fundamental.

Our proof of Theorem \ref{pc} will make use of several new
results and techniques involving stochastic domination, which
we believe are of independent interest and wider applicability.
Stochastic domination by i.i.d.\ processes is a powerful
technique for proving results of this kind, because it enables
facts proved for the i.i.d.\ case to be transferred to other
settings.  One widely used tool is the result of \cite{LSS}
that a $k$-dependent Bernoulli process with sufficiently high
marginals dominates any given i.i.d.\ product measure.
However, the key process that we will need to control (of ``bad
points'') is not $k$-dependent, and in fact is not dominated by
any product measure.  Therefore the methods we use are of a
different nature.

Background on stochastic domination may be found in \cite[Ch.\ II,
\S2]{liggett}, for example. For our purposes, the following definition via
coupling will suffice. Let $X$ and $Y$ be random variables taking values in
the same partially ordered space. Then we say that $X$ \df{stochastically
dominates} $Y$ if there exist $X',Y'$ on some probability space with $X'$ and
$X$ equal in law, $Y'$ and $Y$ equal in law, and $X'\geq Y'$ almost surely.
The underlying partial order will be inclusion (in the case of random sets)
or pointwise ordering (for real functions).

Our first tool is a simple but useful stochastic domination
result on overlapping intervals in a one-dimensional setting.
For $c\in(0,1)$, say that a random variable $X$ has
\df{geometric distribution} with parameter $c$, denoted
$\Geom(c)$, if $\P(X=r)=(1-c)c^r$ for $r=0,1,2,\dots$. (Note
that the value $0$ is included, and that $c$ is the probability
of a ``failure'' rather than a ``success'').  In the following,
the interval $(a,b)$ is taken to be empty if $a=b$.
\begin{theorem}[One-dimensional domination]\label{1d}\sloppypar
Let $(G_n)_{n\in\Z}$ be i.i.d.\ $\Geom(c)$ random variables.
The random set $\Z\cap\bigcup_{n\in\Z}\bigl(n-G_n,n+G_n\bigr)$
 is stochastically
dominated by the open set of i.i.d.\ site percolation on $\Z$
with parameter $\min(4\surd c,1)$.
\end{theorem}

Our second tool concerns first-passage percolation.  As we
explain in Section~\ref{outline}, it can be regarded as
unifying and generalizing ideas in
\cite{BaccelliFoss,fontes-newman,GH-lip}.
 Let $V$ be a countable vertex set.  For every
pair of distinct vertices $x,y\in V$, the directed edge
$e=(x,y)$ is assigned a random \df{passage time} $W(e)=W(x,y)$
taking values in $[0,\infty]$. The passage times of different
edges are independent but not necessarily identically
distributed.  (We can model a process on a graph other than the
complete graph by taking some passage times to be $\infty$
almost surely.)  In addition, each vertex $x\in V$ has a
deterministic \df{source time} $t(x)\in(-\infty,\infty]$ at
which it is ``switched on''. (For example, to model growth
started at a single source $a$ we would take $t(a)=0$ and
$t(x)=\infty$ for all other $x$.) The \df{occupation time} of
$x\in V$ is the time it is first reached:
\[
T(x):=\inf_{
\substack{y_0,y_1,\ldots, y_m: \\ y_m=x}}
\biggl\{
t(y_0) + \sum_{k=1}^{m} W(y_{k-1}, y_k)
\biggr\}.
\]

We now consider a collection of countably many models on the
same vertex set, indexed by $i\in\cI$. Different models have
identically distributed passage times, and are independent of
each other, but may have different source times.  Write
$t_i(x)$ for the source time of $x$ in model $i$,
and $T_i(x)$ for the occupation time of vertex $x$ in this model.
Let
$$\tT(x):=\inf_i T_i(x).$$
Finally, consider another model with source times given by
$$t(x):=\inf_i t_i(x),\quad x\in V,$$
and with the same passage time distributions as the other models.
Write $T(x)$ for the occupation time of $x$ in this model.

\begin{theorem}[First-passage percolation domination]\label{FPP}
Under the above assumptions, $(\tT(x))_{x\in V}$ is
stochastically dominated by $(T(x))_{x\in V}$.
\end{theorem}

We prove Theorems \ref{1d} and \ref{FPP} at the end of the
article. In the next section we explain how these results are
used in the proof of Theorem~\ref{pc}.

\section{Outline of Proof}
\label{outline}

In this section we explain the main ideas behind the proof of
Theorem~\ref{pc}.  Our starting point is the following
strengthening of a result of \cite{DDGHS} (the latter has been
applied in \cite{DDS,GH-sphere}, and extended in other
directions in \cite{GH-lip}). For
$x=(x_1,\ldots,x_{d-1})\in\Z^{d-1}$ and $z\in\Z$ we denote
their concatenation thus:
$(x,z):=(x_1,\ldots,x_{d-1},z)\in\Z^d$.  Vertices of $\Z^d$ will sometimes be called sites.

\newpage
\begin{theorem}[Stacked Lipschitz surfaces]\label{stack}
Consider site percolation on $\Z^d$ with $d\geq 2$.  If the
parameter $p$ is sufficiently close to $1$ then a.s.\ there
exist (random) functions $L_n:\Z^{d-1}\to\Z$, indexed by
$n\in\Z$, with the following properties.
\begin{subequations}
\begin{align}
\label{Lcond1}
&\text{The site $(x, L_n(x))\in\ZZ^d$ is open for all $x\in\ZZ^{d-1}$
 and $n\in\ZZ$.}\\
\label{Lcond2}
&\text{For each $n$, the function $L_n$ is $1$-Lipschitz in the sense that}\\
\nonumber
&\hspace{2cm}|L_n(x)-L_n(x')|\leq 1 \text{ whenever } \|x-x'\|_\infty = 1.\\
\label{Lcond3}
&\text{$L_n(x)>2n$ for all $x$ and $n$.}\\
\label{Lcond4}
&\text{$L_{n-1}(x)<L_{n}(x)$ for all $x$ and $n$.}
\end{align}
\end{subequations}
\end{theorem}
For each $n$, the graph $\big\{(x,L_n(x)): x\in\ZZ^{d-1}\big\}$
of $L_n$ is a ``Lipschitz surface'', and Theorem~\ref{stack}
asserts the existence of an ordered stack of disjoint open
Lipschitz surfaces.  See Figure~\ref{all-2d}. This strengthens
the result of \cite{DDGHS} that one such surface exists for $p$
sufficiently close to $1$.  Our Lipschitz surfaces differ from
those in \cite{DDGHS,GH-lip} in that we use the $\infty$-norm
rather than the $1$-norm in \eqref{Lcond2} -- this is
relatively unimportant, but will be convenient for our
construction.  The condition \eqref{Lcond3} will be helpful in
keeping track of the typical position of each surface.
\begin{figure}
\noindent\makebox[\textwidth]{%
\includegraphics[width=0.7\textwidth]{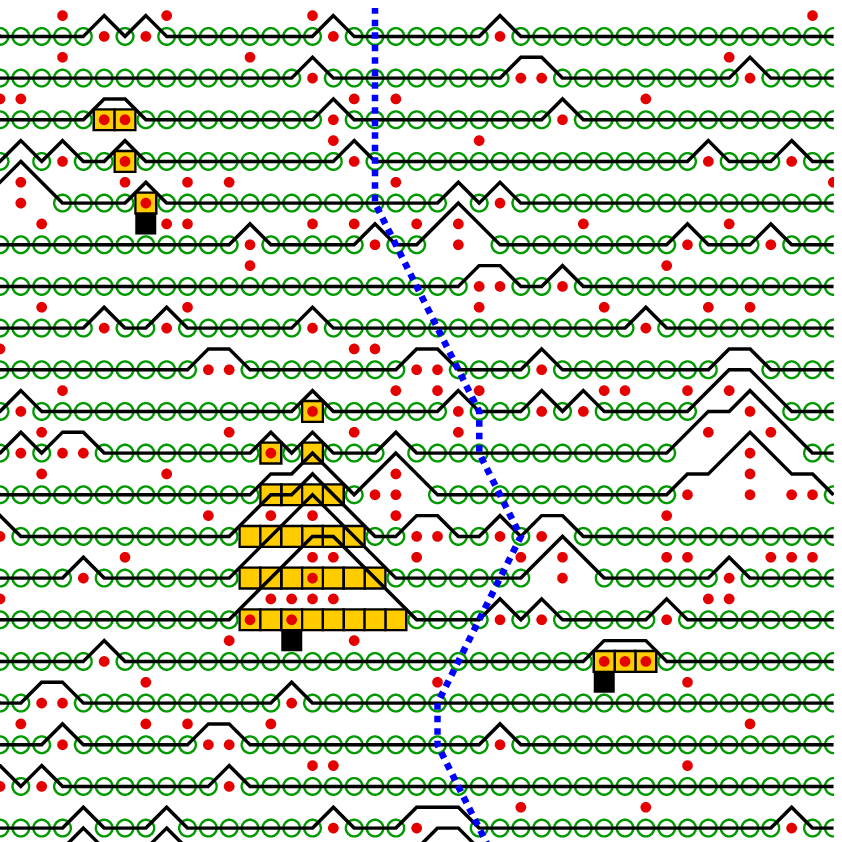}
\begin{minipage}[b][0.7\textwidth][c]{4.3cm}
\begin{tabularx}{4.3cm}{c>{\RaggedRight}X}
\includegraphics{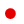}&closed sites;\\[1ex]
\includegraphics{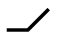}
&stacked Lipschitz surfaces $L_n$ avoiding closed sites;\\[1ex]
\includegraphics{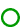}
&good sites (where $L_n$ is as low as possible);\\[1ex]
\raisebox{-1ex}{\includegraphics{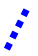}}
&perpendicular Lipschitz surface $H$ avoiding bad sites;\\[1ex]
\includegraphics{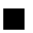}&three selected sites $y$;\\[1ex]
\includegraphics{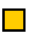}&obstacles $A_y$ at those sites $y$.
\end{tabularx}
\end{minipage}
\hspace{-6mm}\ }
\caption{The main objects used in the construction of the embedding.
Here $d=2$ and $p=0.885$.}
\label{all-2d}
\end{figure}

In proving Theorem \ref{stack} we will define a particular
family of functions $(L_n)$ having additional desirable
properties.  In fact, $(L_n)$ will be the minimal family
satisfying \eqref{Lcond1}--\eqref{Lcond4} in the sense that for
any other such family $(L_n')$ we have $L_n(x)\leq L_n'(x)$ for
all $x,n$.

Our aim is to weave these Lipschitz surfaces together using
another Lipschitz surface perpendicular to the stack. Observe
that the minimum possible value of $L_n(x)$ is $2n+1$.  We pay
particular attention to those positions where this minimum is
attained. Let $\Zeven$ be the set of even integers and $\Zodd$
the set of odd integers. We call the site
$(x,2n+1)\in\ZZ^{d-1}\times\Zodd$ \df{good} if $L_n(x)=2n+1$,
and otherwise \df{bad}. Note that these definitions apply only
to sites whose last coordinate is odd, and depend on the choice
of the functions $L_n$.  Since $(x,L_n(x))$ is always open,
every good site is open.

\begin{theorem}[Perpendicular Lipschitz surface]\label{backbone}
Fix $d\geq 2$.  For $p$ sufficiently close to $1$, the
functions $L_n$ of Theorem ~\ref{stack} may be chosen so that
almost surely there exists a function
$H:\Z^{d-2}\times\Zodd\to\Z$ with the following properties.
\begin{subequations}
\begin{align}
\label{Hcond1}
&\text{The site }
(H(u), u)
\text{ is good for all }u\in\ZZ^{d-2}\times\Zodd.\\
\label{Hcond2}
&
|H(u)-H(u')|\leq 1
\text{ whenever:} \\
&\hspace{2cm} |u_{d-1}-u'_{d-1}|\leq 2,
\text{ and }|u_i-u'_{i}|\leq 1\text{ for }i\leq d-2.\nonumber
\end{align}
\end{subequations}
\end{theorem}

See Figure~\ref{all-2d}.  The set $\{(H(u),u):{u\in\Z^{d-2}\times\Zodd}\}$
forms a kind of Lipschitz surface perpendicular to the $1$ coordinate
direction.  (The ``${\leq}2$'' in \eqref{Hcond2} reflects the appearance of
$\Zodd$ in the domain of $H$.  Note that the $d-1$ coordinate of $u$ becomes
the $d$ coordinate of $(H(u),u)$.) It is relatively straightforward to check
that any functions $L_n$ and $H$ satisfying \eqref{Lcond1}--\eqref{Lcond4}
and \eqref{Hcond1}--\eqref{Hcond2} give rise to an embedding of $\comb^d$ in
the open set of $\Z^d_{[2]}$, as required for Theorem~\ref{pc}. This is
verified in Section \ref{sec:embed}; the function $H$ gives the backbone of
the comb, while the $L_n$'s give the fins. Therefore our main task is to
prove Theorems \ref{stack} and \ref{backbone}.

We will prove Theorem~\ref{stack} via an extension of the methods of
\cite{DDGHS}: the Lipschitz surfaces will be constructed as duals to paths of
a certain type, called $\Lambda$-paths.  Now, since the property
\eqref{Hcond2} required for $H$ is essentially property \eqref{Lcond2} of our
Lipschitz function $L_0$ (modulo a change of coordinate system), an appealing
idea is to try to deduce Theorem~\ref{backbone} from Theorem~\ref{stack}. The
problem, of course, is that the process of good sites is not i.i.d.  It is
also not dominated by any i.i.d\ process (because a vertical column of $k$
consecutive closed sites gives rise to a bad set with volume of order $k^d$).
Nonetheless, we will indeed deduce Theorem~\ref{backbone} from
Theorem~\ref{stack}, using stochastic domination in more subtle ways.

We will proceed by re-expressing the process of bad sites. For
each $y\in\Z^{d-1}\times\Zeven$ we will define a random finite
set $A_y$, called the obstacle at $y$, in such a way that
$$\{x:\text{$x$ is bad}\}=\bigcup_y A_y.$$
The field of obstacles will have the stationarity property that
$(A_y+z)_y$ is equal in law to $(A_{y+z})_y$ for all $z$.  The
obstacle at $y$ will be the set of points that can be reached
from $y$ by $\Lambda$-paths satisfying certain conditions.

The random sets $A_y$ will not be independent of each other
(since the paths used in their construction are shared between
different $y$'s). However, we will prove that they can be
replaced with independent sets in the following sense. Let
$(\tilde{A}_y)_{y}$ be mutually independent random sets, with
$\tilde{A}_y$ equal in law to $A_y$ for each $y$. We will show
\begin{equation}\label{obstacle-domination}
\bigcup_y A_y \text{ is stochastically dominated by }\bigcup_y \tilde{A}_y.
\end{equation}
A similar fact was proved in \cite{GH-lip} in the context of a
Lipschitz percolation model.  An analogous property for a
continuum percolation model was obtained in
\cite{BaccelliFoss}, and related ideas appeared earlier in
\cite{fontes-newman}. We will prove \eqref{obstacle-domination}
by expressing $A_y$ in terms of a first-passage percolation
model (via the paths involved in its definition), and appealing
to the much more general Theorem~\ref{FPP}.

Our task is now reduced to proving the existence of a Lipschitz
surface $H$ (as in \eqref{Hcond2}) avoiding a collection of
independent sets $\tilde{A}_y$.  The ideas behind the proof of
Theorem~\ref{stack} will easily show that the radius around $y$
of the random obstacle $A_y$ (and thus $\tilde{A}_y$) has
exponential tails for $p$ sufficiently close to $1$.  However,
for $d\geq 2$, this is not enough to allow domination of
$\bigcup_y \tilde{A}_y$ by an i.i.d.\ percolation process,
since there the probability of a closed ball of radius $r$
decays exponentially in $r^d$.

The final ingredient is a deterministic observation which
allows us to reduce to a one-dimensional process and hence
overcome the above dimensionality problem. Here it is important
that the object we seek is a Lipschitz surface.  For $x\in\Z^d$
and $r>0$ define the \df{ball} $B(x,r):=\{z\in\Z^d:
\|x-z\|_\infty<r\}$ and the one-dimensional \df{stick}
$S(x,r):=\{x+ae_d: a\in\Z \text{ and } |a|<r\}$.
\begin{lemma}[Balls and sticks]\label{sticklemma}
Suppose $h:\Z^{d-1}\to\Z$ is $1$-Lipschitz (i.e.\
$|h(x)-h(x')|\leq 1$ whenever $\|x-x'\|_\infty\leq 1$). If the
graph $\{(x,h(x)):x\in\Z^{d-1}\}$ does not intersect the stick
$S(y,2r-1)$ then it does not intersect the ball $B(y,r)$.
\end{lemma}

\begin{figure}
\centering
\begin{tikzpicture}[thick,scale=0.35]
\draw[fill=blue!10,rounded corners] (4.7,-.3) rectangle (9.3,4.3);
\draw[fill=blue!30,rounded corners] (6.7,-2.3) rectangle (7.3,6.3);
\draw[rounded corners] (4.7,-.3) rectangle (9.3,4.3);
\draw[fill] (7,2) circle(.15); \draw
(1,3)--(2,2)--(7,7)--(8,7)--(9,6)--(11,6); \tikzstyle{every
node}=[shape=circle,fill=black,inner sep=0,minimum size=0.75mm]
\foreach \x in {1,...,11} \foreach \y in {-2,...,7} \node at
(\x,\y) {};
\end{tikzpicture}
\caption{A Lipschitz surface avoids a ball provided it avoids a stick.}
\label{stickpic}
\end{figure}

See Figure~\ref{stickpic} for an illustration.  Using
Lemma~\ref{sticklemma}, it suffices to construct a Lipschitz
surface that avoids a union of sticks $\bigcup_{x\in\Z^d}
S(x,G_x)$ with i.i.d\ geometric sizes $G_x$. This union
consists of independent one-dimensional processes in each
vertical line.  Therefore we can use Theorem~\ref{1d} to
dominate it by an i.i.d.\ percolation process (with parameter
that tends to $0$ as $p\to 1$), and deduce
Theorem~\ref{backbone} from Theorem~\ref{stack}, and hence
complete the proof of Theorem~\ref{pc}.

In the next four sections we carry out the steps outlined above
to prove Theorem~\ref{pc}.  The stacked surfaces $L_n$ are
constructed in Section~\ref{sec:stack}. Obstacles are defined
and dominated by independent sets in Section~\ref{sec:obs}, and
their radii are bounded in Section~\ref{radiussection}.  The
remaining details (including the stick argument) are completed
in Section~\ref{sec:embed}. Finally we prove the general
domination results, Theorems~\ref{1d} and \ref{FPP}, in
Sections~\ref{avoidingsection}~and~\ref{FPPsection}
respectively. The first-passage percolation result is proved
via dynamic coupling.  For the one-dimensional domination
result we employ a queueing interpretation.

\section{Stacked Lipschitz surfaces}
\label{sec:stack}

In this section we prove Theorem~\ref{stack}.  We first
construct the functions $L_n$, and then prove that they have
the required properties.  We sometimes refer to the positive
and negative senses of the $d$ coordinate as up and down
respectively, and the other coordinates as horizonal.

Define a \df{$\Lambda$-path} to be a sequence of sites $z(0),
z(1), \dots, z(m)\in\Z^d$ such that for each $i<m$,
\begin{equation}\label{Lambdapathdef}
z(i+1)-z(i)\in\{e_d\}\cup\Delta,
\end{equation}
where
$$\Delta:=
\biggl\{-e_d+\sum_{i=1}^{d-1}\alpha_ie_i:
(\alpha_1,\alpha_2,\dots,\alpha_{d-1})\in\{-1,0,1\}^{d-1}\biggr\}.
$$
That is, each step is up or down, but the down-steps may also
be diagonal; there are $3^{d-1}$ different types of down-step
since each of the first $d-1$ coordinates is allowed to remain
the same or change by 1 in either direction.  (Our definition
of a $\Lambda$-path differs slightly from that in \cite{DDGHS},
where only $2d+1$ types of down-step were allowed.  The
difference reflects our use of the $\infty$-norm in
\eqref{Lcond2}.) For an integer $r\geq0$, we call a
$\Lambda$-path \df{$r$-open} if its up-steps have distinct
locations, and at most $r$ of them end with an open site, i.e.\
among the indices $i<m$ for which $z(i+1)-z(i)=e_d$, the sites
$z(i+1)$ are all distinct, and at most $r$ of them are open. We
write $y\tor z$ if there is an $r$-open $\Lambda$-path from $y$
to $z$.

Now define the random set of sites $S_n$ by
\begin{equation}\label{Sndef}
S_n:=\bigl\{z: y\tor z \text{ for some }
r\text{ and some }y\text{ with }y_d=2(n-r)\bigr\}.
\end{equation}
Then let $L_n$ be the function whose graph lies just above
$S_n$:
\begin{equation}\label{Lndef}
L_n(x):=\min\bigl\{\ell\in \Z: (x,\ell)\not\in S_n\bigr\},
\end{equation}
(where $\min\emptyset:=\infty$).

\begin{proposition}\label{Lnlemma}
Let the functions $L_n$ be defined as above.  If $p$ is
sufficiently close to $1$, then a.s.\ $L_n(x)<\infty$ for all
$n$ and $x$, and the properties \eqref{Lcond1}--\eqref{Lcond4}
in Theorem~\ref{stack} all hold.
\end{proposition}

\begin{proof}
From the definition and the underlying stationarity of the
percolation process, the process $(L_n(x))_{n,x}$ is stationary
in the sense that $(L_{n+k}(x+y)-2k)_{n,x}$ has the same law
for any $k\in\ZZ$ and $y\in\ZZ^{d-1}$. Hence for the first
claim it is enough to show that $L_0(0)<\infty$ a.s.

In fact we will show that $L_0(0)$ has exponential tails. For
$h>0$, we have $\PP(L_0(0)>h)=\PP((0,h)\in S_0)$, and this is
at most the expected number of $r$-open $\Lambda$-paths from
the hyperplane $\Z^{d-1}\times \{-2r\}$ to $(0,h)$, summed over
all $r$. For such a path, let $C$ be the number of up-steps
that end in a closed site, let $U$ be the number of up-steps
that end in an open site, and let $D$ be the number of
down-steps (including diagonal steps). Since the path is from
$\Z^{d-1}\times \{-2r\}$ to $(0,h)$ we must have
$C+U-D=h-(-2r)$, i.e.\ $U=D-C+h+2r$. Since the path is $r$-open
we have $U\leq r$, or equivalently $A\geq 0$ where $A:=r-U$.

For given $U,C,D$, the number of ways to choose a
$\Lambda$-path ending at $(0,h)$ together with an assignment of
states open and closed to its up-steps is at most $K^{U+C+D}$,
where $K:=3^{d-1}+2$.  (There are $3^{d-1}$ possible directions
for a down-step, and two possible states for an up-step).  For
any such choice, the probability that the chosen states match
the percolation configuration is $p^U q^C \leq q^C$, where
$q:=1-p$.

Therefore
\begin{align*}
\P(L_0(0)>h)&\leq \sum_{\substack{U,C, D,r\geq 0: \\ C+U-D=h+2r,\\U\geq r}}
K^{U+C+D} q^C \\
&\leq \sum_{A,D,r\geq 0} K^{2D+h+2r} q^{D+h+r+A}\\
&= \big(Kq\big)^h \sum_{A\geq 0} q^A
\sum_{D\geq0}\big(K^2q\big)^D
\sum_{r\geq0}\big(K^2q\big)^r,
\end{align*}
which converges (exponentially fast) to $0$ as $h\to\infty$
whenever $q<K^{-2}$.  (For the second inequality above, we
rewrote $U$ and $C$ in terms of $A$ and dropped the conditions
$U\geq 0$ and $C\geq 0$.)

Now we verify properties (\ref{Lcond1})--(\ref{Lcond4}). For
(\ref{Lcond1}), observe that, for some $y,r$ as in the
definition of $S_n$, there is an $r$-open path to the site
$(x,L_n(x)-1)$, but there is none to the site $(x, L_n(x))$.
Thus the site $(x, L_n(x))$ must be open -- if it were closed,
the $r$-open path to $(x,L_n(x)-1)$ could be extended one step
upward (or else it already passed through that site).

Next, note that from the definition of $S_n$, if  $z\in S_n$
then $z+v \in S_n$  for all $v\in\Delta$. This gives the
Lipschitz property for $L_n$ as required for (\ref{Lcond2}).
For (\ref{Lcond3}) note that certainly $(x,2n)\in S_n$ for all
$x$ and $n$. Finally, if $z\in S_n$ then $z+e_d \in S_{n+1}$,
giving (\ref{Lcond4}).
\end{proof}

\begin{proof}[Proof of Theorem~\ref{stack}]
This is immediate from Proposition~\ref{Lnlemma} above.
\end{proof}

\section{Obstacles}
\label{sec:obs}

In this section we define obstacles, and show that they can be
dominated by independent versions.  Let the functions $L_n$ be
defined as in \eqref{Lndef}. As mentioned earlier, we say that
$$\text{site }(x,2n+1)\in\ZZ^{d-1}\times\Zodd
\text{ is \df{good} if } L_n(x)=2n+1,$$
and otherwise it is \df{bad}.  For $y\in\ZZ^{d-1}\times\Zeven$
we define the \df{obstacle at} $y$ to be
\begin{equation}\label{Aydef}
A_y:=\Bigl\{z\in\ZZ^{d-1}\times\Zodd:
y\xrightarrow{\frac{z_d-y_d-1}{2}}z\Bigr\}.
\end{equation}
Note that $A_y$ is defined only for $y$ of even height, while
it consists of a set of sites of odd heights.
\begin{lemma}\label{obstacles}
We have
$$\Bigl\{x\in \ZZ^{d-1}\times\Zodd: x\text{\rm\ is bad}\Bigr\} =
\bigcup_{y\in\ZZ^{d-1}\times\Zeven} A_y.
$$
\end{lemma}

\begin{proof}
From \eqref{Sndef},\eqref{Lndef} it follows that
$z=(x,2n+1)\in\ZZ^{d-1}\times\Zodd$ is bad if and only if $z\in
S_n$, which in turn is equivalent to the existence of
$y\in\ZZ^{d-1}\times\Zeven$ such that $
y\xrightarrow{(z_d-y_d-1)/2}z. $
\end{proof}

Now let $(\tA_y)_{y\in\ZZ^{d-1}\times\Zeven}$ be mutually
independent random sets, with $\tA_y$ equal in law to $A_y$ for
each $y$.

\begin{proposition}
\label{independentobstacleprop} With the above definitions,
$\bigcup_y A_y$ is stochastically dominated by $\bigcup_y
\tA_y$.
\end{proposition}
\begin{proof}
We rephrase the definition of obstacles in terms of a
first-passage percolation model. Let each upward directed edge
$(z, z+e_d)$, $z\in\ZZ^d$  have passage time $1$ if $z+e_d$ is
open, and 0 if $z+e_d$ is closed.  Each downward directed edge
$(z, z+v)$, $v\in\Delta$ has passage time $0$.  All other edges
have passage time $\infty$. Note that all the passage times are
independent.  We assign source time $t(y)=y_d/2$ to each site
$y\in\ZZ^{d-1}\times\Zeven$, and source time $\infty$ to all
sites in $\ZZ^{d-1}\times\Zodd$.
From the definition of $r$-open paths, for
$z\in\ZZ^{d-1}\times\Zodd$, we have
\begin{equation}\label{bada}
z\in\bigcup_y A_y \iff T(z)\leq \frac{z_d-1}{2}.
\end{equation}

Now consider a countable family of models indexed by
$y\in\ZZ^{d-1}\times\Zeven$. All models have the same
distribution of passage times as described above, and are
independent of each other, but in model $y$, the only source is
$y$, with $t(y)=y_d/2$ (all other sites have source time
$\infty$).  Write $T_y(z)$ for the passage time to $z$ in model
$y$.
The set of sites $z\in\ZZ^{d-1}\times\Zodd$ with $T_y(z)\leq
(z_d-1)/2$ has the same law as $A_y$; let us define it to be
$\tA_y$. Thus the family $(\tA_y)$ has precisely the
distribution required. Writing $\tT(z):=\inf_y T_y(z)$ and
using \eqref{bada}, we have
\begin{equation}\label{badb}
z\in\bigcup_y \tA_y \iff \tT(z)\leq \frac{z_d-1}{2}.
\end{equation}

Theorem~\ref{FPP} tells us that $(\tT(z))$ is stochastically
dominated by  $(T(z))$. Using (\ref{bada}) and (\ref{badb}),
this implies that $\bigcup_y A_y$ is stochastically dominated
by $\bigcup_y \tA_y$, as required.
\end{proof}

\section{Radii of obstacles}\label{radiussection}

Let $R_y$ be the \df{radius} of the obstacle at $y$, by which
we mean the smallest $r$ such that $A_y\subseteq B(y,r)$
(recall that $B(y,r):=\{z\in\Z^d:\|y-z\|_\infty<r\}$). So
$R_y=0$ if and only if $R_y$ is empty.  Since all sites in
$A_y$ must have $d$ coordinate strictly greater than $y_d$, we
observe that
 $R_y$ is never equal to $1$, and also that $A_y\subseteq
 B(y+e_d,R_y)$.
Recall that our geometric random variables are supported on the
non-negative integers.

\begin{lemma}
\label{obstacleradiuslemma} If $p$ is sufficiently close to $1$
then for each $y\in\ZZ^{d-1}\times\Zeven$, the radius $R_y$ of
the obstacle at $y$ is stochastically dominated by a $\Geom(c)$
random variable, where $c=c(p)\to0$ as $p\to1$.
\end{lemma}

\begin{proof}
Since as observed above, $R_y$ never takes the value $1$, it
will be enough to show that $\PP(R_y>r)<c^{r+1}$ for all $r\geq
1$. We use a path-counting argument similar to that already
used in the proof of Proposition~\ref{Lnlemma}.

Suppose that $R_y>r$.  Then by the definition of $A_y$ there
exists $z$ with $\|y-z\|_\infty\geq r$ and
$y\xrightarrow{(z_d-y_d-1)/2}z$.  Consider some
$(z_d-y_d-1)/2$-open $\Lambda$-path from $y$ to such a $z$, and
as before let it have $U$ up-steps ending in open sites, $C$
up-steps ending in closed sites, and $D$ (diagonal- or)
down-steps. Since the path is $(z_d-y_d-1)/2$-open we have
$$U\leq(z_d-y_d-1)/2=(U+C-D-1)/2,$$
and so $C-U-D-1\geq 0$.  Since $\|y-z\|_\infty\geq r$, either
$z_d-y_d\geq r$, in which case $U+C-D\geq r$, or else $z$ and
$y$ differ by at least $r$ in some other coordinate, in which
case $D\geq r$ (since only down-steps permit horizontal
movement).  Using the earlier inequality, in either case we
have $U+C-r\geq 0$.

As before, let $K:=3^{d-1}+2$ and $q:=1-p$.  Then, bounding via
the expected number of paths,
\begin{align*}
\P(R_y>r)&\leq
\sum_{\substack{U,C,D\geq 0: \\ C-U-D-1\geq 0 \\ U+C-r\geq 0}}
K^{U+C+D} q^C\\
&\leq \sum_{X,Y,D\geq 0}
K^{Y+D+r} q^{(X+Y+D+r+1)/2}\\
&= (K\surd q)^r\surd q \sum_{X\geq 0} \surd q^X
\sum_{Y\geq 0} (K\surd q)^Y \sum_{D\geq 0} (K\surd q)^D,
\end{align*}
where in the second inequality we wrote $X:=C-U-D-1$ and
$Y:=U+C-r$ and dropped the conditions $U,C\geq 0$.  The last
expression equals
$$\frac{\surd q}{(1-\surd q)(1-K\surd q)^2} \;
(K\surd q)^r=A a^r, \text{say,}$$ (where $A=A(K,q)$ and
$a=a(K,q)$ are defined by the last equality), provided $a<1$.
Finally, we have $A a^r\leq \max(A,a)^{r+1}$, and $\max(A,a)\to
0$ as $q\to 0$.
\end{proof}

\section{Completing the embedding}
\label{sec:embed}

In this section we conclude the proof of Theorem~\ref{pc} by
combining the various ingredients together with some geometric
arguments.  We start by proving Lemma~\ref{sticklemma}, which
states that balls may be replaced with sticks for the purposes
of finding a Lipschitz function that avoids them.
\begin{proof}[Proof of Lemma~\ref{sticklemma}]
For $z\in\Z^d$ we write $\widehat{z}:=(z_1,\ldots,z_{d-1})$, so
$z=(\widehat{z},z_d)$.  Let $r\geq 1$ (otherwise the ball and
stick in the lemma are both empty). Suppose that
$\{(x,h(x)):x\in \Z^{d-1}\}$ does intersect $B(y,r)$, say at
the site $u=(\widehat u,h(\widehat u))\in B(y,r)$.  Thus
$\|\widehat u-\widehat y\|_\infty\leq r-1$ and $|h(\widehat
u)-y_d|\leq r-1$.  By the Lipschitz property, the former
implies $|h(\widehat u)-h(\widehat y)|\leq r-1$. Therefore
$|h(\widehat y)-y_d|\leq 2r-2$.  Thus $(\widehat y,h(\widehat
y))\in S(y,2r-1)$, and $\{(x,h(x)):x\in\Z^{d-1}\}$ intersects
$S(y,2r-1)$.
\end{proof}

Next we check that the Lipschitz surfaces of
Theorems~\ref{stack} and \ref{backbone} can be combined to give
an embedding of the comb.  A slightly subtle point in
dimensions $d\geq 3$ is that the backbone surface $H$ will not
typically ``line up'' with the stacked surfaces $L_n$ with
respect to the intermediate coordinates $2,\ldots,d-1$.
Nevertheless, the use of the $\infty$-norm in the definitions
of $\Z^d_{[2]}$ gives enough wiggle room to permit an
embedding. Suppose we are given the functions $L_n$ and $H$.
For $z\in\ZZ^d$, define $x(z)\in \ZZ^{d-1}$ and $f(z)\in\ZZ^d$
by
\begin{gather}\label{xdef}
x(z):= \Bigl( z_1+H\bigl( (z_2,z_3,\ldots,z_{d-1}, 2z_d+1)
\bigr),z_2,\ldots,z_{d-1} \Bigr);
\\
 \label{fdef} f(z):=\Bigl( x(z), L_{2z_d+1}\big(x(z)\big)
\Bigr).
\end{gather}

\begin{lemma}\label{Hlemma}
Suppose that the functions $L_n$ and $H$ satisfy the conditions
of Theorems~\ref{stack} and \ref{backbone}.  Then the function
$f$ defined above is an embedding of the comb $\comb^d$ in
the open set of $\Z^d_{[2]}$.
\end{lemma}

\begin{proof}[Proof of Lemma \ref{Hlemma}]
First observe that the site $f(z)$ is open for all $z$; this is
immediate from \eqref{fdef} and property \eqref{Lcond1} of
$L_n$.

We next check that $f$ is injective.  From \eqref{Lcond4}, the
sites $(x, L_n(x))$ and $(x', L_{n'}(x'))$ are distinct
whenever $x\ne x'$ or $n\ne n'$. Suppose $z\neq z'$.  If $z$
and $z'$ differ in any of the first $d-1$ coordinates then by
\eqref{xdef}, $x(z)\ne x'(z)$, while if they differ in the $d$
coordinate then $2z_d+1\neq 2z'_d+1$. Thus \eqref{fdef} gives
$f(z)\ne f(z')$, as required.

To show that $f$ is an embedding of the comb it remains to
check that
\begin{equation}\label{fcheck}
\|f(z)-f(z+e_i)\|_\infty\leq 2
\end{equation}
for all $z\in\Z^d$ and $i\leq d-1$, and also for $i=d$ whenever
$z_1=0$.  We first note the following key point. If $z_1=0$,
then
\begin{equation}\label{altfdef}
f(z)=\bigl(x(z), 2z_d+1\bigr).
\end{equation}
This is because, by (\ref{Hcond1}), for $u=(z_2, z_3, \dots,
z_{d-1}, 2z_d+1)$, the site $\big(H(u), u\big)$ is good, which
means that
\begin{align*}
L_{2z_d+1}(x(z))&=
L_{2z_d+1}(H(u), z_2, z_3, \dots, z_{d-1})
\\
&=2z_d+1,
\end{align*}
so that the two expressions in (\ref{fdef}) and (\ref{altfdef})
are the same.

We now verify that \eqref{fcheck} holds in the cases claimed.
First suppose that $z$ and $z'$ differ by $1$ in the $i$th
coordinate, where $i\leq d-1$, and that all the other
coordinates agree.  By the Lipschitz property \eqref{Hcond2} of
$H$, the first coordinates of $x(z)$ and $x(z')$ differ by at
most $1$, and clearly the same is true of the other
coordinates.  Hence by property \eqref{Lcond2}, we have
$|L_{2z_d+1}(x(z))-L_{2z_d+1}(x(z'))|\leq 1$. It follows that
$\|f(z)-f(z')\|_\infty\leq 1<2$.

Now suppose that $z$ and $z'$ differ by $1$ in the last
coordinate, and all the other coordinates agree, and suppose in
addition that $z_1=0$. Thus $2z_d+1$ and $2z'_d+1$ differ by
$2$, and by \eqref{altfdef}, $f(z)=(x(z), 2z_d+1)$ and
$f(z')=(x(z'), 2z'_d+1)$.  By \eqref{Hcond2} and \eqref{xdef},
the first coordinates of $x(z)$ and $x(z')$ differ by at most
$1$, and the other coordinates agree. Therefore
$\|f(z)-f(z')\|_\infty=2$ as required.
\end{proof}

We are ready to prove Theorem \ref{backbone} and deduce Theorem
\ref{pc}.

\begin{proof}[Proof of Theorem \ref{backbone}]
We need to show that if $p$ is sufficiently close to $1$ there
exists $H$ satisfying \eqref{Hcond1} and \eqref{Hcond2}.

Condition (\ref{Hcond1}) says that the surface $\{(H(u),u)\}$
must avoid every obstacle $A_y$, and by
Proposition~\ref{independentobstacleprop}, for this it suffices
to instead find a surface avoiding the independent obstacles
$\tA_y$. As remarked in Section~\ref{radiussection} we have
$A_y\subseteq B(y+e_d,R_y)$ (and $y+e_d\in
\Z^{d-1}\times\Zodd$), and by Lemma~\ref{obstacleradiuslemma},
$R_y$ is dominated by a geometric random variable whose
parameter $c=c(p)$ can be made as small as desired by taking
$p$ large enough. Therefore it remains to show that for $c$
sufficiently small there exists $H$ satisfying \eqref{Hcond2}
such that $\{(H(u),u):u\in \Z^{d-2}\times\Zodd\}$ avoids
$\bigcup_{y\in\Z^{d-1}\times\Zodd} B(y,G_y)$, where $(G_y)$ are
i.i.d.\ $\Geom(c)$.  Note that now all the relevant sites have
odd heights.

Now we map $\Z^{d-1}\times\Zodd$ to $\Z^d$ via the
transformation $(m,v,2n+1)\mapsto(v,n,m)$, for $v\in\Z^{d-2}$
and $m,n\in\Z$. It thus suffices to find a function
$h:\Z^{d-1}\to\Z$ satisfying the same $1$-Lipschitz condition
\eqref{Lcond2} as $L_0$, and whose graph
$\{(x,h(x)):x\in\Z^{d-1}\}$ avoids $\bigcup_{y\in\Z^d}
B(y,G_y)$ for $(G_y)$ i.i.d.\ $\Geom(c)$. (The transformation
does not increase $\infty$-norms).

Now we apply Lemma~\ref{sticklemma}.  The graph of $h$ will
avoid the balls $B(y,G_y)$ provided it avoids the sticks
$S(y,(2G_y-1)_+)$.  Observe also that if $G$ is $\Geom(c)$ then
$(2G-1)_+$ is dominated by a $\Geom(c')$ random variable, where
$c'=\surd c$, so it suffices to avoid
$\mathcal{S}:=\bigcup_{y\in\Z^d} S(y,G_y')$ where $(G_y')$ are
i.i.d.\ $\Geom(c')$.

The random set $\mathcal{S}$ consists of independent components
in each of the lines $\{x\}\times \Z$, for $x\in \Z^{d-1}$.
Within any such line, Theorem~\ref{1d} shows that it is
stochastically dominated by the open set of an i.i.d.\
percolation process with parameter $c''=4\surd c'$. Thus the
whole set $\mathcal{S}$ is dominated by the open set of an
i.i.d.\ percolation process with parameter $c''$ on $\Z^{d}$.
Hence it follows from Theorem~\ref{stack} (exchanging the roles
of open and closed sites) that there exists a function $h$
satisfying our requirements if $c''$ is sufficiently small.
Since $c''=4 \, c(p)^{1/4}$, this holds provided $p$ is
sufficiently close to $1$.
\end{proof}

\begin{proof}[Proof of Theorem~\ref{pc}]
This is immediate from Lemma~\ref{Hlemma} and
Theorems~\ref{stack} and \ref{backbone}.
\end{proof}

\begin{proof}[Proof of Corollary~\ref{large-m}]
Fix $k\geq 1$ and call the site $x\in\Z^d$ \df{occupied} if the
cube $kx+[0,k)^d$ contains some open site in the percolation
model.  For any graph $G$, if there exists an embedding of $G$
in the {\em occupied} sites of $\Z_{[m]}^d$ then there exists
an embedding of $G$ in the open sites of $\Z_{[km+k-1]}^d$: we
simply choose one open site from the cube of each occupied site
in the image. Therefore,
$$\Bigl[1-p_c(G,\Z_{[km+k-1]}^d)\Bigr]^{k^d} \geq
 1-p_c(G,\Z_{[m]}^d).$$
Setting $m=2$ and $G=\comb^d$, and using the fact that
$p_c(G,\Z_{[M]}^d)$ is decreasing in $M$, the result follows
from Theorem~\ref{pc}.
\end{proof}

\section{First-passage percolation domination}
\label{FPPsection}

In this section we prove Theorem~\ref{FPP}.  Recall that we
have a collection of models indexed by $i$, and an additional
model whose source times are given by infima of source times of
the others.  Write $W_i(e)$ and $W(e)$ for the passage time of
edge $e$ in model $i$ and in the additional model respectively.

\begin{proof}[Proof of Theorem~\ref{FPP}]
The argument is most straightforward in the case where the
vertex set $V$ and the index set $\mathcal{I}$ are finite, and
where a.s.\ the occupation times $T_i(x)$ are finite and
distinct for all $i$ and $x$. (This property holds, for
example, when all the source times are distinct and finite, and
each edge passage time is either $\infty$ or some positive
continuous random variable). We begin with this case, and then
extend to the general case by a limiting argument.

We will define the collection of passage times $(W(e))$ as a
function of the collection $(W_i(e))$, in such a way that
$W(e)$ shares the common distribution of the $W_i(e)$, that the
passage times $W(e)$ are independent for different $e$, and
that $\tT(x)\leq T(x)$ for all $x$. This explicit coupling
implies the stochastic domination required.  For a directed
edge $(x,y)$ we set
$$W(x,y):= W_i(x,y),\text { where $i$
minimizes }T_i(x).$$

First we aim to show that $\tT(x)\leq T(x)$.
We have
\[
T(x)=\min_{
\substack{y_0,y_1,\dots, y_m \\ y_m=x}}
\biggl\{
t(y_0) + \sum_{k=1}^{m} W(y_{k-1}, y_k)
\biggr\}.
\]
If $y_0,y_1,\dots, y_m$ is a minimizing path in the above
expression, then for all $r$ with $0\leq r\leq m$,
\[
T(y_{r})=
t(y_0)+\sum_{k=1}^{r} W(y_{k-1}, y_k),
\]
and in particular $T(y_{r})=T(y_{r-1})+W(y_{r-1}, y_r)$ for
$1\leq r\leq m$.

We will show by induction that $\tT(y_{r})\leq T(y_{r})$ for
all such $r$. For the $r=0$ case, we have
\[
T(y_0)=t(y_0)=\min_i t_i(y_0) \geq \min_i T_i(y_0) = \tT(y_0).
\]
Now suppose $\tT(y_{r-1})\leq T(y_{r-1})$. Let $i$ minimize
$T_i(y_{r-1})$, so that $W(y_{r-1}, y_r)=W_i(y_{r-1}, y_r)$,
and $T(y_r)\geq \tT(y_r)=T_i(y_r)$. Then
\begin{align*}
T(y_r)
&=T(y_{r-1})+W(y_{r-1}, y_r)\\
&\geq T_i(y_{r-1})+W_i(y_{r-1}, y_r)\\
&\geq T_i(y_r) \geq \tT(y_r),
\end{align*}
completing the induction.

It remains to show that the $W(e)$ as defined are indeed
independent with the required distributions.  We will do this
by giving a different construction of all the models.  The idea
is to run them simultaneously in real time, revealing the
random passage times only when they are needed.

First consider a single model $i$.  We begin by choosing all
the passage times $W_i(e)$, with the correct distributions, but
we do not yet reveal them.  (We can think of them as written on
cards associated with the edges, which will be turned over at
the appropriate times). \df{Label} each vertex with a time by
which it needs to be examined; initially these are just the
source times.  Now we repeatedly do the following.  Find the
vertex $x$ with the earliest (smallest) label among those that
have not yet been examined. Then \df{examine} $x$, which is to
say, reveal the passage times $W_i(x,y)$, $y\in V$ of all edges
leading out of $x$, and relabel each vertex $y\neq x$ with the
minimum of: its current label, and the label at $x$ plus
$W_i(x,y)$.  Repeat until all vertices have been examined. It
is clear that the vertices are examined in order of their
occupation times $T_i(x)$, and that when a vertex is examined
it is labeled with its occupation time (and this label does not
subsequently change). Our assumptions guarantee that these
times are all distinct, and so the choice of which vertex to
examine next is always unambiguous.  (These claims may be
checked formally by induction over the vertices in order of
their occupation times).

Now consider simultaneously running all the models
$i\in\mathcal{I}$ in the way just described.  We first choose
all the passage times independently, without revealing them. At
each step we examine the unexamined vertex with the earliest
label across all the models (and we examine it only in the
minimizing model). Clearly each individual model evolves
exactly as before (but with its steps interspersed with the
others).  Our assumptions guarantee that no two steps are
simultaneous. Finally we construct the passage times $W(e)$ of
the additional model: at each step, if the vertex that is
examined (say vertex $x$ in model $i$) is the first to be
examined among the copies of that vertex $x$ in all the models,
then we in addition set $W(x,y)=W_i(x,y)$ for all $y\in V$.
Since the label of vertex $x$ in model $i$ at this step is
$\min_i T_i(x)$ ($=\tT(x)$), this agrees with the earlier
definition of $W(x,y)$.  The key point is that the decision to
assign $W_i(x,y)$ to $W(x,y)$ is made before the value of
$W_i(x,y)$ is revealed.  It follows that the passage times
$(W(e))$ assigned to the additional model are independent and
have the correct distributions, as required.

To extend to the general case, we consider a sequence of
approximating finite systems of the kind just considered.
Without loss of generality, suppose that the index set $\cI$ is
contained in $\NN$. In the $n$th system in our sequence of
approximations, we take a finite vertex set $V^{(n)}$, such
that $V^{(n)}\uparrow V$ as $n\to\infty$. All source times and
passage times involving a vertex not in $V^{(n)}$ are set to
infinity. Any source time for a vertex in $V^{(n)}$ that was
previously be set to infinity is now instead given value $n$
(to ensure that every vertex in $V^{(n)}$ is reached in finite
time). Furthermore we consider only models indexed by
$i\in\{1,2,\dots,n\}$.  Finally we perturb the source times of
vertices in $V^{(n)}$, and the passage times of edges joining
points of $V^{(n)}$, by adding an independent Uniform$(0,\tfrac
1n)$ random variable to each. (This means that the source times
are no longer deterministic -- however, we can regard the
randomness as being on two levels: given any choices of the
source times, we have a set of models with random passage
times).  This ensures that a.s., the finite system satisfies
all of our earlier assumptions.

Write $T^{(n)}_i(x)$, $\tT^{(n)}(x)$ and $T^{(n)}(x)$ for the
passage times in the $n$th approximation. These quantities are
finite for any $x\in V_n$. But also, since any set of vertices
is eventually contained within $V_n$, and each model $i$ is
eventually included in the system, it follows that for any
finite set $A$, the random vectors $(\tT^{(n)}(x))_{x\in A}$
and $(T^{(n)}(x))_{x\in A}$ converge in distribution as
$n\to\infty$ to $(\tT(x))_{x\in A}$ and $(T(x))_{x\in A}$
respectively.  We know from the argument applied to the finite
case that $(\tT^{(n)}(x))_{x\in A}$ is stochastically dominated
by $(T^{(n)}(x))_{x\in A}$ for any $n$. Hence from the
convergence in distribution as $n\to\infty$, we obtain also
that in fact $(\tT(x))_{x\in A}$ is stochastically dominated by
$(T(x))_{x\in A}$.  Since this holds for any finite subset
$A\subseteq V$, it follows that in fact $(\tT(x))_{x\in V}$ is
stochastically dominated by $(T(x))_{x\in V}$, as desired.
\end{proof}

We remark that Theorem~\ref{FPP} may easily be extended for
example to models with undirected edges instead of (or in
addition to) directed edges, or with passage times at sites. As
long as all the passage times are independent, exactly the same
methods used above will continue to apply.

\section{Domination in one dimension}
\label{avoidingsection}

In this section we prove Theorem~\ref{1d}.  We start with the
following one-sided version.  The interval $[a,b)$ is taken to
be empty if $a=b$.

\begin{proposition}\label{1side}
For $c\in(0,1)$, let $(G_n)_{n\in\Z}$ be i.i.d.\ $\Geom(c)$
random variables. The random set
$\Z\cap\bigcup_{n\in\Z}[n,n+G_n)$ is stochastically dominated
by the open set of i.i.d.\ site percolation on $\Z$ with
parameter $\min(2c,1)$.
\end{proposition}

\begin{proof}
Define the indicator variable
\begin{align*}
B_i&:=\ind\bigl[i\in [n,n+G_n)\text{ for some } n\in\Z\bigr]
\\
&\;=\ind\bigl[G_n>i-n \text{ for some }n\leq i\bigr].
\end{align*}
Then we must prove that $(B_i)_{i\in Z}$ is dominated by an
i.i.d.\ Bernoulli($2c$) sequence (when $2c\leq 1$). For this,
it is enough to show that a.s.,
\begin{equation}\label{condeq}
\PP\big(B_i=1 \bigmid (B_j)_{j<i}\big)\leq 2c
\end{equation}
for all $i$ (see e.g.\ \cite[Lemma 1]{russo}). Since the
process $(B_i)$ is stationary it suffices to show
\eqref{condeq} for $i=0$.

We may think of the system as an $M/M/\infty$ queue in discrete
time. At each time $n$, a customer arrives whose service time
is $G_n$. The customer will depart at time $n+G_n$, and so
occupies the system during the interval $[n,n+G_n)$. (If
$G_n=0$ then the customer is never seen at all). Now $B_i$ is
the indicator of the event that there is some customer present
at time $i$.

We introduce the key random variable
$$N:=\max\{n\geq0 :G_{-n}\geq n\}.$$
We can think of $N$ as the age of the oldest customer who has
not left the system before time $0$. (For example, if $N=0$
then all the previous customers have left before time $0$;  on
this event we have $B_0=\ind[G_0>0]$, since the only customer
who could be present at time $0$ is the one arriving at time
$0$.)  The Borel-Cantelli lemma shows that $N$ is a.s.\ finite.

We claim that a.s.
\begin{equation}
\label{condclaim}
\PP\big(B_0=1\bigmid(B_j)_{j<0}, N\big)
=\PP\big(B_0=1 \bigmid N\big).
\end{equation}
To prove this, we must establish
\begin{equation}\label{condclaim_e}
\PP\big(B_0=1\bigmid E, \;N=n\big)
=\PP\big(B_0=1 \bigmid N=n\big),
\end{equation}
for all events $E\in\sigma((B_j)_{j<0})$ for which the
conditional probability on the left exists. Observe first that
$N=n$ forces $B_{-n}=\cdots=B_{-1}=1$, so it is enough to prove
\eqref{condclaim_e} for $E\in\sigma((B_j)_{j<-n})$.

To verify the above, observe that the two families
$\mathcal{L}:=(G_j)_{j<-n}$ and $\mathcal{R}:=(G_j)_{j\geq -n}$
are conditionally independent of each other given $N=n$: this
is because they are independent without the conditioning, while
$\{N=n\}$ is the intersection of an event in
$\sigma(\mathcal{L})$ and an event in $\sigma(\mathcal{R})$.
Now, $(B_j)_{j<-n}$ is a function of $\mathcal{L}$. On the
other hand, on $N=n$, we have that $B_0$ is a function of
$\mathcal{R}$ (since $N=n$ guarantees that any customer
arriving before time $-n$ has already left the system before
time 0). It follows that $(B_j)_{j<-n}$ and $B_0$ are
conditionally independent given $N=n$; this gives precisely
\eqref{condclaim_e} for $E\in\sigma((B_j)_{j<-n})$ as required,
thus proving the claim \eqref{condclaim}.

Returning to the proof of \eqref{condeq}, we have
\begin{align*}
\lefteqn{\PP(B_0=1\mid N=n)}
\\
\;\;&=\PP\bigl(G_{-j}>j \text{ for some } j\geq 0
\bigmid
G_{-n}\geq n, \text{ and } G_{-j}< j \text{ for all } j>n\bigr)
\\
&=\PP\bigl(G_{-j}>j \text{ for some } 0\leq j\leq n
\bigmid
G_{-n}\geq n\bigr)
\\
&=\PP(G_{-n}>n \mid G_{-n}\geq n)
\\
&\qquad+\bigl[1-\PP(G_{-n}>n \mid G_{-n}\geq n)\bigr]\;
\PP\bigl(G_{-j}>j \text{ for some } 0\leq j < n\bigr)
\\
&\leq c+(1-c)(c+c^2+\dots+c^n)\leq 2c.
\end{align*}
Combining with \eqref{condclaim}, we have shown that
$\PP(B_0=1\mid\big(B_j)_{j<0}, N)\leq 2c$ a.s.; then averaging
over $N$ gives (\ref{condeq}) (for $i=0$), as required.
\end{proof}

\begin{proof}[Proof of Theorem~\ref{1d}]
Let $(L_n)_{n\in\Z}$ and $(R_n)_{n\in\Z}$ be i.i.d.\
$\Geom(\surd c)$ random variables.  Let $G_n:=\min(L_n,R_n)$,
which is $\Geom(c)$.  Then, with all intervals understood to
denote their intersections with $\Z$,
$$\bigcup_{n\in\Z} (n-G_n,n+G_n) \subseteq
 \Bigl(\bigcup_{n\in\Z} (n-L_n,n]\Bigr) \cup
 \Bigl(\bigcup_{n\in\Z} [n,n+R_n)\Bigr).$$
By Proposition~\ref{1side} and symmetry, the right side is
dominated by the union of the open sets of two independent
percolation processes each of parameter $2\surd c$, which is
itself the open set of a percolation process of parameter
$1-(1-2\surd c)^2\leq 4\surd c$.
\end{proof}

\bibliographystyle{habbrv}
\bibliography{cp}

\end{document}